\documentclass{article}

\usepackage{amsfonts,amssymb,amsmath,amsthm,amscd}
\usepackage[matrix,arrow,curve]{xy}

\usepackage[colorlinks=true]{hyperref}

\newcommand{\Spec}{\operatorname{Spec}}

\newcommand{\diffSpec}{\operatorname{Spec^\Sigma}}
\newcommand{\PSpec}{\operatorname{PSpec}}
\newcommand{\Rad}{\operatorname{Rad}}
\newcommand{\perfRad}{\operatorname{Rad^{perf}}}
\newcommand{\PId}{\operatorname{Id^\Sigma}}

\newcommand{\PRad}{\operatorname{Rad^\Sigma}}
\newcommand{\Max}{\operatorname{Max}}
\newcommand{\Id}{\operatorname{Id}}

\newcommand{\PMax}{\operatorname{Max^\Sigma}}

\newcommand{\Qt}{\operatorname{Qt}}

\newcommand{\Fun}{\operatorname{F}}

\newtheorem{theorem}{Theorem}

\newtheorem{corollary}[theorem]{Corollary}

\newtheorem{proposition}[theorem]{Proposition}
\newtheorem*{statement*}{Statement}
\newtheorem*{proposition*}{Proposition}
\newtheorem*{theorem*}{Theorem}
\newtheorem*{lemma*}{Lemma}
\newtheorem*{fact*}{Fact}

\theoremstyle{definition}

\newtheorem*{definition*}{Definition}

\newtheorem*{example*}{Example}

\theoremstyle{remark}
\newtheorem{remark}[theorem]{Remark}
\newtheorem*{remark*}{Remark}

\author{D.\,V.~Trushin\footnote{This work was partially supported by the NSF
grants CCF-0964875 and CCF-0952591}}
\title{Difference Nullstellensatz\footnote{\Xy-pic package is used}}
\date{}

\begin{document}

\maketitle

\begin{abstract}
We prove different forms of the Nullstellensatz for difference
fields and absolutely flat simple difference rings, called
pseudofields. A difference ring is a ring on which an arbitrary
group is acting by means of ring automorphisms.
\end{abstract}

\section{Introduction}

We develop a basis for a geometric theory of difference equations.
The latter means that we want the solutions of difference equations
to be points of some ``affine space''. Such a theory should be used
in the Galois theory of differential and difference equations with
difference parameters. A first example of such a theory with its
application to the Galois theory of difference equations can be
found in~\cite{Tr5} and~\cite{AOT}, respectively. We plan to apply
the results of this paper to develop a general Galois theory with
difference parameters. It should be noted that, in the case of one
difference operator, there is another approach, which is being
developed in~\cite{WibmerDcChev}.

One could proceed in the following two different ways. The first
approach comes from the ACFA theory (see a detailed survey
in~\cite{ChadzSur} and also~\cite{ChatzHrushMDF}), where we consider
only solutions in fields. A differential Galois theory based on
difference fields can be found in~\cite[Chapter~8]{LevinDA}. The
second approach says that we should extend the class of fields to a
wider class. Usually, people use the class of finite products of
fields
(see~\cite{AmMasPV,AOT,ChatHarSinDGG,HarSinLDiff,TakeuHopPV,Tr5,vdPSinDcE,vdPSinLDcE,WibmerThesis}).
It turns out that, for the needs of geometric theory, finite
products of fields are not good enough. Our main result is that we
have found the desired class of rings called pseudofields. And the
main theorem says that every difference pseudofield (in particular,
difference field) embeds into a difference closed pseudofield, that
is an analogue of an algebraically (differentially) closed field in
commutative (differential) algebra (Theorem~\ref{theorps}). This
result allows to consider all necessary solutions of any set of
difference equations in a one fixed pseudofield.

Our difference closed field (see our general definition in
Section~\ref{sec21}) should be considered as a generalization of a
model of the ACFA theory. For such fields, we prove two versions of
the Nullstellensatz, that are Theorem~\ref{theorf} and~\ref{mainth}.
Moreover, we show that our results hold even in the case of an
arbitrary set of difference indeterminates (Theorems~\ref{theorefi}
and~\ref{mainthi} for fields, and~\ref{theorii} for pseudofields,
respectively).

It should be noted that the most important thing was to find the
right definition of a pseudofield. Earlier, people used only finite
products of fields to study solutions of difference equations. But,
if we want to study solutions of ``all possible equations'', we
should use another structure. Our solution is using absolutely flat
rings, which are a direct generalization of finite products of
fields. Also, we note that absolutely flat difference rings have
been considered in~\cite{HrushvN}. Moreover, this new class of rings
requires to develop absolutely new machinery, which is presented in
the paper. We notice that the key role in our paper is played by
pseudoprime ideals and the ring of functions presented in
Section~\ref{sec22}. Pseudoprime ideals in different forms appeared
in~\cite{Tr5,WibmerDcChev,WibmerThesis}. The ring of functions was
used in~\cite{MorikDG1,UmMorDG2}. In the latter papers, the authors
use slightly different terminology.

The paper is organized as follows. In Section~\ref{sec2}, we give
all necessary definitions and recall some basic facts.
Section~\ref{sec3} is devoted to discussion of the subject of the
paper. We explain the meaning of the difference closedness for
difference rings and show how the general point of view gives us the
notions of difference closed fields and pseudofields. The rest of
the paper consists of two parts devoted to fields and pseudofields,
respectively. In Sections~\ref{sec4} and \ref{sec5}, we prove that
every difference field embeds into a difference closed field
(Theorem~\ref{theorf}) and show a strong version of the difference
Nullstellensatz for difference closed fields (Theorem~\ref{mainth}).
The case of infinitely many difference indeterminates is given
separately. In Section~\ref{sec6}, we deal with pseudofields. We
prove that every pseudofield embeds into a difference closed one
(Theorem~\ref{theorps}). The case of infinitely many indeterminates
is also presented in the last section.

\paragraph{Acknowledgment.} I am grateful to Alexey
Ovchinnikov for his helpful suggestions, to Alexander Levin for the
interest to the paper and help in learning difference algebra, and
to Michael Wibmer for a conversation of the material.

\section{The detailed structure of the paper}

Here, we describe some technical details of the paper.

Section~\ref{sec2} is devoted to basic notions and facts. In
Section~\ref{sec21}, we give all necessary definitions. We underline
that pseudoprime ideals and their relation with the usual spectrum
play the key role in our paper. Precisely, we use the continuous map
$\pi$ from the spectrum to the pseudospectrum. This relation reduces
the difference problem to a commutative one. In Section~\ref{sec22},
we recall the notion of the ring of functions. Since we do not have
a pseudofield of fractions we have to use Proposition~\ref{taylor}
allowing to embed an arbitrary simple difference ring into a
pseudofield.

Section~\ref{sec3} is divided into two parts. In the first one, we
discuss the notion of difference closed ring and a weak form of the
difference Nullstellensatz for fields (Proposition~\ref{prop_eqB})
and pseudofields (Propositions~\ref{prop_eqA}).  As in the second
part, we extend the subject to an arbitrary set of difference
indeterminates (Proposition~\ref{prop_eqB_inft} for fields and
Propositions~\ref{prop_eqA_inft} for pseudofields, respectively).

Section~\ref{sec4} is devoted to difference fields. Our important
technical tool is Proposition~\ref{genlemma} allowing to deal with
non-Noetherian rings (this is a generalization of the Hilbert basis
theorem). In Section~\ref{sec41}, we show that every difference
field embeds into a difference closed one (Theorem~\ref{theorf}).
Since we have some problems with compatibility, we have to slightly
modify Artin's proof of the similar fact for fields. We use only
tensor products of compatible integral domains.  In
Section~\ref{sec42}, we show that every difference field embeds even
into a $\kappa$-closed difference field (Theorem~\ref{theorefi}).
This fact has absolutely the same proof but the statements of these
theorems are slightly different. A difference closed field is a
natural generalization of a model of the ACFA to the case of an
arbitrary set of difference operators. Section~\ref{sec5} is devoted
to a strong version of the difference Nullstellensatz for fields.
Section~\ref{sec51} deals with finitely many difference
indeterminates (Theorem~\ref{mainth}), and Section~\ref{sec52} with
infinitely many ones (Theorem~\ref{mainthi}). To prove these
theorems, we are able to use the usual localization because we deal
with perfect difference ideals.

In Section~\ref{sec6}, we extend our theory to the case of
pseudofields. This case is much harder than the case of difference
fields. In particular, if we deal with fields, we can use the
Hilbert basis theorem and the localization. But, if we deal with
pseudofields, we have to produce absolutely new machinery.
Proposition~\ref{poset} is an analogue of the Ritt-Raudenbush basis
theorem. To prove Proposition~\ref{poset}, we use a bound on the
number of prime ideals in a pseudofield (Proposition~\ref{specst}).
To apply the Artin proof for the existence of a difference closed
pseudofield containing a given one, we show that the tensor product
operation over a pseudofield has no ``zero-divisors''
(Proposition~\ref{tens}). As above, Section~\ref{sec61} is devoted
to the case of finitely many difference indeterminates
(Theorem~\ref{theorps}), and Section~\ref{sec62} to the case of
infinitely many ones (Theorem~\ref{theorii}).  It should be noted
that, at present, we do not have an analogue of the strong version
of the Nullstellensatz for pseudofields. Partially, this happens
because we cannot use the localization.

\section{Terms and notation}\label{sec2}

\subsection{Definitions}\label{sec21}

Let $\Sigma$ be an arbitrary group. This group will be fixed during
the paper. A ring $A$ will be said to be a difference ring if $A$ is
an associative commutative ring with an identity element such that
the group $\Sigma$ is acting on $A$ by means of ring automorphisms.
A difference homomorphism is a homomorphism preserving the identity
element and commuting with the action of $\Sigma$. A difference
ideal is an ideal stable under the action of the group $\Sigma$,
that is $\sigma(\mathfrak a)\subseteq \mathfrak a$ for all
$\sigma\in\Sigma$. Sometimes, we will write $\Sigma$ instead of the
word difference. A simple difference ring is a difference ring such
that the zero ideal is its unique proper difference ideal. For every
ideal $\mathfrak a\subseteq A$ and every element $\sigma\in\Sigma$,
the image of $\mathfrak a$ under $\sigma$ will be denoted by
$\mathfrak a^\sigma$. A difference ring $A$ will be called a
pseudofield if $A$ is an absolutely flat simple difference ring.

It should be noted that we obtain usual difference rings if $\Sigma$
is a free finitely generated abelian group. Moreover, since we
consider the action of a group instead of a monoid, all our rings
are inversive. In particular, all our difference ideals are
reflexive in sense of~\cite{CohnDA} and~\cite{LevinDA}.

The set of all, radical, prime, maximal ideals of $A$ will be
denoted by $\Id A$, $\Rad A$, $\Spec A$, $\Max A$, respectively. The
set of all difference ideals of $A$ will be denoted by $\PId A$.
$\PRad A$ will denote the set of all radical difference ideals of
$A$.

The definition of a pseudoprime ideal was given in~\cite{Tr5} in the
most general case (the definition in some particular case was given
in~\cite{WibmerThesis}). Let $S\subseteq A$ be a multiplicatively
closed subset in a difference ring $A$ and let $\mathfrak q$ be a
maximal $\Sigma$-ideal not meeting $S$. In this case, the ideal
$\mathfrak q$ will be called a pseudoprime ideal. The set of all
pseudoprime ideals will be denoted by $\PSpec A$ and will be called
the pseudospectrum of $A$. The set of all maximal difference ideals
of $A$ will be denoted by $\PMax A$. Every maximal difference ideal
is pseudoprime. For we should take $S=\{1\}$. For an ideal
$\mathfrak a$ and a multiplicatively closed set $S$, the ideal
\[
\{x\mid \exists s\in S: xs\in\mathfrak a \}
\]
will be denoted by $\mathfrak a:S^\infty$.

For every ideal $\mathfrak a\subseteq A$, the largest $\Sigma$-ideal
contained in $\mathfrak a$ will be denoted by $\mathfrak a_\Sigma$.
Such an ideal exists because it coincides with the sum of all
difference ideals contained in $\mathfrak a$. Note that
\[
\mathfrak a_\Sigma=\{\, a\in\mathfrak a\mid\forall
\sigma\in\Sigma\colon \sigma(a)\in\mathfrak a \,\}.
\]
So, we have the following map
\[
\pi\colon \Id A\to \PId A
\]
defined by $\mathfrak a\mapsto \mathfrak a_\Sigma$. A
straightforward calculation shows that, for every family of ideals
$\mathfrak a_\alpha$, we have
\[
\pi(\bigcap_{\alpha} \mathfrak a_\alpha)=\bigcap_{\alpha}
\pi(\mathfrak a_\alpha).
\]
We see that, for any ideal $\mathfrak a$, we have
\[
\mathfrak a_\Sigma=\bigcap_{\sigma\in\Sigma} \mathfrak a^\sigma.
\]

Note that the restriction of $\pi$ to the spectrum gives the map
\[
\pi\colon \Spec A\to \PSpec A.
\]
The ideal $\mathfrak p$ will be called $\Sigma$-associated with
pseudoprime $\mathfrak q$ if $\pi(\mathfrak p)=\mathfrak q$. Let
$\mathfrak q$ be a pseudoprime ideal and let $S$ be a
multiplicatively closed set from the definition of $\mathfrak q$,
then every prime ideal containing $\mathfrak q$ and not meeting $S$
is $\Sigma$-associated with $\mathfrak q$. So, the map $\pi\colon
\Spec A\to \PSpec A$ is surjective.

If $S$ is a multiplicatively closed subset in $A$, then the  ring of
fractions of $A$ with respect to $S$ will be denoted by $S^{-1}A$.
If $S=\{t^n\}_{n=0}^\infty$, then $S^{-1}A$ will be denoted by
$A_t$. If $\mathfrak p$ is a prime ideal of $A$ and $S=A\setminus
\mathfrak p$, then  $S^{-1}A$ will be denoted by $A_{\mathfrak p}$.
The residue field of a prime ideal $\mathfrak p$ is the field of
fractions of $A/\mathfrak p$ and is denoted by $k(\mathfrak p)$. If
$A$ is an integral domain, then the field of fractions of $A$ will
be denoted by $\Qt(A)$.

For every subset $X$ of a difference ring $A$, the smallest
difference ideal containing $X$ will be denoted by $[X]$ and the
smallest perfect difference ideal containing $X$ will be denoted by
$\{X\}$~\cite[Chapter~2, Section~2.3, Definition~2.3.1]{LevinDA}.
The radical of an ideal $\mathfrak a$ will be denoted by $\mathfrak
r(\mathfrak a)$. The inclusion $\mathfrak r([X])\subseteq\{X\}$
always holds. The smallest ideal generated by the set $X$ will be
denoted by $(X)$.

The set of all perfect difference ideals of a ring $A$ will be
denoted by $\perfRad A$ and the set of prime difference ideals will
be denoted by $\diffSpec A$.

Let $f\colon A\to B$ be a homomorphism of rings and let $\mathfrak
a$ and $\mathfrak b$ be ideals of $A$ and $B$, respectively. We
define the extension $\mathfrak a^e$ of $\mathfrak a$ to be the
ideal $f(\mathfrak a)B$ generated by $f(\mathfrak a)$ in $B$. The
ideal $f^*(\mathfrak b)=f^{-1}(\mathfrak b)$ is called the
contraction of $\mathfrak b$ and is denoted by $\mathfrak b^c$. If a
homomorphism $f\colon A\to B$ is difference and ideals $\mathfrak a$
and $\mathfrak b$ are difference, then their extensions and
contractions, respectively, are difference.

Let $f\colon A\to B$ be a $\Sigma$-homomorphism of difference rings
and let $\mathfrak q$ be a pseudoprime ideal of $B$. The contraction
$\mathfrak q^c$ is pseudoprime because $\pi$ is surjective and
commutes with $f^*$. So, we have a map from $\PSpec B$ to $\PSpec
A$. This map will be denoted by $f^*_\Sigma$. It follows from
definition that the following diagram is commutative
\[
\xymatrix{
    \Spec B\ar[r]^{f^*}\ar[d]^{\pi}&\Spec A\ar[d]^{\pi}\\
    \PSpec B\ar[r]^{f^*_\Sigma}&\PSpec A\\
}
\]

The ring of difference polynomials $A\{Y\}$ over $A$ in variables
$Y$ is the polynomial ring $A[\Sigma Y]$ (that is the polynomial
ring in variables denoted by $\sigma y$ for all $\sigma\in \Sigma$
and $y\in Y$), where $\Sigma$ is acting in the natural way ($\sigma
(\tau y)=(\sigma\tau)y$). A difference ring $B$ will be called a
difference $A$-algebra if there is a difference homomorphism $A\to
B$. It is clear that every difference $A$-algebra can be presented
as a quotient ring of a difference polynomial ring over $A$.

For an arbitrary difference ring $A$ and an arbitrary set $Y$, an
arbitrary difference polynomial $f\in A\{Y\}$ involves only finitely
many indeterminates. Therefore, $f$ can be considered as a function
on $A^Y$. Thus, we can speak about zeros of difference polynomials
in the ``affine space'' $A^Y$ and about polynomials vanishing on
some subset. If $E\subseteq A\{Y\}$, then $V(E)\subseteq A^Y$ will
denote the set of all common zeros of $E$ in $A^Y$. And, for an
arbitrary subset $X\subseteq A^Y$, we set $I(X)$ to be the set of
all difference polynomials vanishing on $X$. We will use this
notation for finite and infinite sets $Y$. For the convenience, the
sets $V(E)$ will be called difference algebraic varieties.

\subsection{The ring of functions}\label{sec22}

Let us recall the definition of the ring $\Fun B$
(see~\cite[Section~3.3]{Tr5}). For an arbitrary commutative ring
$B$, the set of all functions from $\Sigma$ to $B$ will be denoted
by $\Fun B$, i.~e., $\Fun B=B^\Sigma$. As a commutative ring, it
coincides with the product $\prod_{\sigma\in\Sigma} B$. The group
$\Sigma$ is acting on $\Fun B$ by the rule
$\sigma(f)(\tau)=f(\sigma^{-1}\tau)$. For every element $\sigma$ of
$\Sigma$, there is the substitution homomorphism
\begin{equation*}
\begin{split}
\gamma_{\sigma}\colon\Fun B&\to B\\
f&\mapsto f(\sigma)
\end{split}
\end{equation*}
It is clear that $\gamma_{\tau}(\sigma
f)=\gamma_{\sigma^{-1}\tau}(f)$.

The following statement is proved in~\cite[Proposition~6]{Tr5}.

\begin{proposition}\label{taylor}
Let $A$ be a difference ring, $B$ is a ring, and $\varphi\colon A\to
B$ is a homomorphism. Then, for every $\sigma\in \Sigma$, there
exists a unique difference homomorphism $\Phi_{\sigma}\colon A\to
\Fun B$ such that the following diagram is commutative
\[
\xymatrix{
                                    & {\Fun B}\ar[d]^{\gamma_{\sigma}} \\
        A\ar[r]^{\varphi}\ar[ur]^{\Phi_\sigma} & B
}
\]
\end{proposition}

\section{The statement of problem}\label{sec3}

In this section, we will discuss different kinds of the difference
Nullstellensatz appearing in difference algebra and prepare some
statements for further use. In commutative algebra, the Hilbert
Nullstellensatz shows a relation between the set of solutions of a
systems of polynomial equations with the set of radical ideals in a
polynomial ring. We will find the similar relation in the difference
case.

\subsection{Difference closed fields and pseudofields}\label{sec31}

Firstly, we consider the case of finitely many difference
indeterminates. Consider the following two assertions about a
difference ring $A$.

\begin{description}\label{def_A}
\item[({\bf A1})] For every natural $n$ and every proper difference
ideal $\mathfrak a\subseteq A\{y_1,\ldots,y_n\}$, there is a common
zero for $\mathfrak a$ in $A^n$.

\item[({\bf A2})] For every natural number $n$ and every subset $E\subseteq
A\{y_1,\ldots,y_n\}$, if there exists a common zero for $E$ in $B^n$
for some ring $B$ containing $A$, then there exists a common zero
for $E$ in $A^n$.
\end{description}

\begin{proposition}\label{prop_diff_simple}
If condition \emph{(}{\bf A1}\emph{)} holds for a difference ring
$A$, then $A$ is simple.
\end{proposition}
\begin{proof}
Suppose not, there is a nontrivial difference ideal $\mathfrak
a\subseteq A$. Then the ideal $\mathfrak a\{y_1,\ldots,y_n\}$ is not
trivial and has no common zeros in $A^n$ for all $n$.
\end{proof}

\begin{proposition}\label{prop_eqA}
For an arbitrary difference ring $A$, conditions \emph{(}{\bf
A1}\emph{)} and \emph{(}{\bf A2}\emph{)} are equivalent.
\end{proposition}
\begin{proof}

({\bf A1})$\Rightarrow$({\bf A2}). Suppose that, for some subset
\[
E\subseteq A\{y_1,\ldots,y_n\},
\]
there is a ring $B$ containing $A$ and a common zero $b\in B^n$ for
$E$. Consider the substitution homomorphism
\[
A\{y_1,\ldots,y_n\}\to B,
\]
where $y_i\mapsto b_i$. Then its kernel is a proper difference ideal
containing $E$. Therefore, the ideal $[E]$ is a proper difference
ideal. Now, by ({\bf A1}), we have a common zero in $A^n$.

({\bf A2})$\Rightarrow$({\bf A1}). Let $\mathfrak a\subseteq
A\{y_1,\ldots,y_n\}$ be a proper difference ideal, the ring
$A\{y_1,\ldots,y_n\}/\mathfrak a$ be denoted by $B$, and the images
of $y_i$ be denoted by $b_i$. Then $b_i$ give a common zero in $B^n$
for $\mathfrak a$. Now, ({\bf A2}) gives a common zero for
$\mathfrak a$ in $A^n$.
\end{proof}

We are interested in difference rings satisfying these equivalent
conditions. It should be noted that we use all difference ideals
$\mathfrak a$ in condition ({\bf A1}) and that we use all difference
rings $B$ in condition ({\bf A2}). We can weaken these conditions.
For ({\bf A1}), we can use only perfect difference ideals and, for
({\bf A2}), we can use only integral domains. Now, suppose that the
ring $A$ is an integral domain. So, we have the following.

\begin{description}\label{def_B}
\item[({\bf B1})] For every natural number $n$ and every proper
perfect difference ideal $\mathfrak a\subseteq A\{y_1,\ldots,y_n\}$,
there exists a common zero for $\mathfrak a$ in $A^n$.

\item[({\bf B2})] For every natural number $n$ and every subset $E\subseteq
A\{y_1,\ldots,y_n\}$, if there is a common zero for $E$ in $B^n$,
where $B$ is a difference integral domain containing $A$, then there
is a common zero for $E$ in $A^n$.
\end{description}

It should be noted that the condition ``$A$ to be an integral
domain'' is somewhat important in these definitions (see the
proposition below). We need it to avoid the following pathologic
situation. It could be that $A$ is not an integral domain and does
not contain any prime difference ideal. Then such a ring satisfies
both definitions. And again, these conditions are equivalent.

\begin{proposition}\label{prop_eqB}
If $A$ is a difference integral domain, then conditions \emph{(}{\bf
B1}\emph{)} and \emph{(}{\bf B2}\emph{)} are equivalent.
\end{proposition}
\begin{proof}

({\bf B1})$\Rightarrow$({\bf B2}). This part of the proof is
absolutely the same as in the proof of Proposition~\ref{prop_eqA}.

({\bf B2})$\Rightarrow$({\bf B1}). Let $\mathfrak t\subseteq
A\{y_1,\ldots,y_n\}$ be a proper perfect difference ideal. By Zorn's
lemma, there exists a maximal perfect difference ideal $\mathfrak p$
containing $\mathfrak t$. The ideal $\mathfrak p$ is prime. Indeed,
let $a\notin \mathfrak p$ be an arbitrary element and  $S$ be the
multiplicatively closed subset generated by elements $\sigma(a)$ for
all $\sigma\in \Sigma$. Since $\mathfrak p$ is perfect, $S\cap
\mathfrak p=\varnothing$. By definition, we have $\mathfrak p
\subseteq \mathfrak p:S^\infty$ and $S\cap \mathfrak
p:S^\infty=\varnothing$. Therefore, $\mathfrak p = \mathfrak
p:S^\infty$. So, every element $a$ not belonging to $\mathfrak p$ is
not a zero devisor modulo $\mathfrak p$. (Another proof for
$\Sigma=\mathbb N^n$ is given in the proof of Proposition~2.3.4
of~\cite[Chapter~2, Section~2.3]{LevinDA}).

Let us denote the quotient ring $A\{y_1,\ldots,y_n\}/\mathfrak p$ by
$B$ and the images of $y_i$ by $b_i$. Then $b_i$ form a common zero
for $\mathfrak t$ in $B^n$. Now, from ({\bf B2}), it follows that
there is a common zero for $\mathfrak t$ in $A^n$.
\end{proof}

We can simply modify condition ({\bf B2}) as follows.

\begin{description}
\item[({\bf B2'})] For every natural number $n$ and every subset $E\subseteq
A\{y_1,\ldots,y_n\}$, if there is a common zero for $E$ in $B^n$,
where $B$ is a difference field containing $A$, then there is a
common zero for $E$ in $A^n$.
\end{description}

In the further text, we will study pseudofields satisfying
conditions ({\bf A1}) and ({\bf A2}) and fields satisfying ({\bf
B1}) and ({\bf B2}). The corresponding pseudofields and fields will
be called difference closed. We will prove that every difference
pseudofield (field) embeds into a difference closed pseudofield
(field).

\subsection{$\kappa$-difference closed fields and
pseudofields}\label{sec32}

In the previous section, we deal with difference polynomials in
finitely many variables. In this one, we will describe a general
notion of $\kappa$-closedness. Here, $\kappa$ is supposed to be a
fixed infinite cardinal.

Let $Y$ be an arbitrary infinite set of cardinality $\kappa$.
Consider the following two assertions about a difference ring $A$.

\begin{description}
\item[({\bf A1})] For every proper difference
ideal $\mathfrak a\subseteq A\{Y\}$, there is a common zero for
$\mathfrak a$ in $A^Y$.

\item[({\bf A2})] For every subset $E\subseteq
A\{Y\}$, if there exists a common zero for $E$ in $B^Y$, where $B$
is a ring containing $A$, then there exists a common zero for $E$ in
$A^Y$.
\end{description}

As in the previous section we have.

\begin{proposition}\label{prop_eqA_inft}
For an arbitrary difference ring $A$, conditions  \emph{(}{\bf
A1}\emph{)} and \emph{(}{\bf A2}\emph{)} are equivalent.
\end{proposition}
\begin{proof}
The proof is absolutely the same as the proof of
Proposition~\ref{prop_eqA}, so, we omit it.
\end{proof}

And using only perfect difference ideals, we get the following
conditions.

\begin{description}
\item[({\bf B1})] For every proper perfect difference ideal
$\mathfrak a\subseteq A\{Y\}$, there exists a common zero for
$\mathfrak a$ in $A^Y$.

\item[({\bf B2})] For every subset $E\subseteq
A\{Y\}$, if there is a common zero for $E$ in $B^Y$, where $B$ is a
difference integral domain containing $A$, then there is a common
zero for $E$ in $A^Y$.
\end{description}

\begin{proposition}\label{prop_eqB_inft}
For an arbitrary difference integral domain $A$, conditions
\emph{(}{\bf B1}\emph{)} and \emph{(}{\bf B2}\emph{)} are
equivalent.
\end{proposition}
\begin{proof}
We omit the proof because it repeats the proof of
Proposition~\ref{prop_eqB}.
\end{proof}

Pseudofields satisfying properties ({\bf A1}) and ({\bf A2}) will be
called $\kappa$-closed pseudofields, and difference fields
satisfying ({\bf B1}) and ({\bf B2}) will be called $\kappa$-closed
difference fields. We will prove that every pseudofield (field)
embeds even in a $\kappa$-closed pseudofield (difference field).

It should be noted that the notion of $\kappa$-closedness is an
analogue of $\kappa$-saturation for models in model theory
(see~\cite[Chapter~9, Section~2]{BPoizotCMT}). Namely,
$\kappa$-closed difference fields are exactly $\kappa$-saturated.
However, pseudofields cannot be described in terms of axioms of the
first order language. Therefore, we obtain absolutely new objects
with absolutely new properties, and $\kappa$-closedness is the
desired analogue of $\kappa$-saturation.

\section{The case of fields}\label{sec4}

In this section, we will deal with fields and perfect and prime
difference ideals only. The cases of finitely and infinitely many
difference indeterminates will be considered separately. However,
the proofs are very similar and we will omit some of them to avoid a
repetition.

\subsection{Finitely many indeterminates}\label{sec41}

Let the countable cardinal be denoted by $\omega$.

\begin{proposition}\label{genlemma}
Let $K$ be a field. Consider a polynomial ring $K[Y]$, where $Y$ is
an arbitrary set of indeterminates. Then every ideal of  $K[Y]$ is
generated by the set of cardinality $|Y|\times\omega$.
\end{proposition}
\begin{proof}

If the set $Y$ is finite, then every ideal is finitely generated by
the Hilbert Basis theorem and, thus, the statement holds. Now, we
suppose that $Y$ is infinite. Then $|Y|\times \omega=|Y|$. We may
suppose that $Y=\{y_\alpha\}_{\alpha\in \mathcal A}$ and the set
$\mathcal A$ is well-ordered (we identify the elements of $\mathcal
A$ with the corresponding ordinals, thus we can speak about the
cardinality of $\alpha\in \mathcal A$). Then the set
$\{\,y_\alpha\in Y\mid y_\alpha< y_\gamma\,\}$ will be denoted by
$Y_\gamma$.

Consider subsets $Y_0,Y_1,\ldots,Y_n,\ldots, Y_\omega$. Let
$\mathfrak a\subseteq K[Y_\omega]$  be an ideal, we set $\mathfrak
a_n=\mathfrak a\cap K[Y_n]$. It is clear that  $\mathfrak
a=\cup_n\mathfrak a_n$. Every ideal $\mathfrak a_n$  is finitely
generated. Consequently, their union is countably generated. In
other words, every ideal of $K[Y_\omega]$ is countably generated.

Using transfinite induction, we will prove that every ideal in
$K[Y_\gamma]$ is generated by the set of cardinality $|\gamma|$.
Consider the case of a successor ordinal. Suppose that all ideals in
$K[Y_\alpha]$ are generated by the set of elements of cardinality
$|\alpha|$. Then the rings $K[Y_\alpha]$ and $K[Y_{\alpha+1}]$ are
isomorphic. Therefore, every ideal in $K[Y_{\alpha+1}]$ is generated
by the set of cardinality $|\alpha|=|\alpha+1|$.

The case of a limit ordinal. Let
$Y_\gamma=\cup_{\alpha<\gamma}Y_\alpha$. For every ideal $\mathfrak
a\subseteq K[Y_\gamma]$,  we set $\mathfrak a_\alpha=\mathfrak a\cap
K[Y_\alpha]$. By the induction hypothesis, the ideals $\mathfrak
a_\alpha$ are generated by the sets $E_\alpha$ such that
$|E_{\alpha}|= |\alpha|$. Then $\mathfrak a$ is generated by the set
$E=\cup_{\alpha<\gamma}E_\alpha$. Using~\cite[Chapter~8,
Section~8.3]{BPoizotCMT}), we get
\[
|E|\leqslant|\bigsqcup_{\alpha<\gamma}E_\alpha|\leqslant|
\bigsqcup_{\alpha<\gamma}\gamma|\leqslant|\gamma|\times|\gamma|=|\gamma|
\]
As we can see, the cardinality of $E$ is not greater than
$|\gamma|$.
\end{proof}

\begin{corollary}\label{generff}
For each difference field $K$, every ideal of the ring
$K\{y_1,\ldots,y_n\}$ is generated by the set of cardinality
$|\Sigma|\times\omega$.
\end{corollary}

Now, we can prove the following result.

\begin{theorem}\label{theorf}
Every difference field $K$ embeds into a difference closed field.
\end{theorem}
\begin{proof}

{\bf Construction}. The set of all difference finitely generated
integral domains over $K$ up to isomorphism will be denoted by
$\mathcal B$.  We construct the following partially ordered set
$\mathbb S$. The elements of the set are the pairs
\[
\left(\mathop{\otimes}_{B\in\mathcal B'}\Qt(B),\mathfrak p_{\mathcal
B'}\right),
\]
where $\mathcal B'\subseteq \mathcal B$ be some subset in $\mathcal
B$ and  $\mathfrak p_{\mathcal B'}$ be a prime difference ideal in
$\otimes_{B\in \mathcal B'}\Qt(B)$ (all tensor products are taken
over $K$). This set $\mathbb S$ is partially ordered with respect to
the following order
\[
\left(\mathop{\otimes}_{B\in\mathcal B'}\Qt(B),\mathfrak p_{\mathcal
B'}\right)\leqslant\left(\mathop{\otimes}_{B\in\mathcal
B''}\Qt(B),\mathfrak p_{\mathcal B''}\right)\Leftrightarrow \mathcal
B'\subseteq\mathcal B'',\: \mathfrak p_{\mathcal
B''}\cap\mathop{\otimes}_{B\in\mathcal B'}\Qt(B)=\mathfrak
p_{\mathcal B'}.
\]

From Zorn's lemma, it follows that there exists a maximal element in
$\mathbb S$. Let
\[
\left(\mathop{\otimes}_{B\in\overline{\mathcal B}}\Qt(B),\mathfrak
p_{\overline{\mathcal B}}\right)
\]
be a maximal element. The residue field of the ideal $\mathfrak
p_{\overline{\mathcal B}}$ will be denoted by $K_1$. Now, we can
repeat this construction with the field $K_1$ instead of $K$ and
will get a field $K_2$, then $K_3$ and so on. Using transfinite
induction, we are able to define the field $K_\alpha$ for every
ordinal $\alpha$. Let $\gamma$ be the first cardinal greater than
$|\Sigma|\times\omega$. The field $K_{\gamma}$ will be denoted by
$L$. We will show that $L$ is a difference closed field.

{\bf Difference closedness}. Let $\mathfrak t$ be a proper perfect
difference ideal in the polynomial ring $L\{y_1,\ldots,y_n\}$. Then
it is contained in some prime difference ideal $\mathfrak p$ (the
proof is as the proof of Proposition~\ref{prop_eqA}). The quotient
ring
\[
L\{y_1,\ldots,y_n\}/\mathfrak p
\]
will be denoted by $R$. By Proposition~\ref{generff}, the ideal
$\mathfrak p$ is generated by the set of cardinality $|\Sigma|\times
\omega$. Let $E$ denote the set of generators of $\mathfrak p$. So,
$\mathfrak p=(E)$, where $|E|\leqslant|\Sigma|\times \omega$. Since
$|\Sigma|\times \omega<\gamma$, all elements of $E$ are defined over
an intermediate subfield $K_\alpha$ with $\alpha<\gamma$ (see the
paragraph before Theorem~8.14 of~\cite[Chapter~8,
Section~4]{BPoizotCMT}).

The difference ring $K_\alpha\{y_1,\ldots,y_n\}/(E)$ will be denoted
by $R_0$. Then
\[
R=L\mathop{\otimes}_{K_\alpha}R_0=L\mathop{\otimes}_{K_{\alpha+1}}K_{\alpha+1}\mathop{\otimes}_{K_\alpha}R_0.
\]
Now, we see that $R_0$ and $K_{\alpha+1}\otimes R_0$ are integral
domains because they are subrings of $R$. We will show that there
exists a difference homomorphism from $R_0$ to $K_{\alpha+1}$ over
$K_{\alpha}$. For that it suffices to find a difference subring in
$K_{\alpha+1}$ isomorphic to $R_0$.

Let $\mathcal B$  be the family of all difference finitely generated
integral domains over $K_\alpha$ up to isomorphism. Then, by the
definition of $K_{\alpha+1}$, we have $\mathcal B=\mathcal B'\sqcup
\mathcal B''$, where
\[
K_{\alpha+1}=\Qt\left(\left(\mathop{\otimes}_{B\in\mathcal
B'}\Qt(B)\right)/\mathfrak p_{\mathcal B'}\right).
\]
The ring $R_0$ belongs to $\mathcal B$ by the definition of
$\mathcal B$. Let us show that $R_0$ belongs to $\mathcal B'$.
Suppose not. Then the pair
\[
\left( \mathop{\otimes}_{B\in\mathcal
B'}B\mathop{\otimes}_{K_\alpha}R_0, \mathfrak p_{\mathcal B'}\otimes
R_0 \right)
\]
is greater than
\[
\left( \mathop{\otimes}_{B\in\mathcal B'}B, \mathfrak p_{\mathcal
B'} \right),
\]
a contradiction.

Now, we have the difference homomorphism
\[
R=L\otimes_{K_{\alpha+1}} K_{\alpha+1} \otimes_{K_\alpha} R_0\to
L\otimes_{K_{\alpha+1}} K_{\alpha+1} = L.
\]
Then the images of the elements $y_1,\ldots,y_n$ give us the desired
common zero for $\mathfrak p$ and thus for $\mathfrak t$.
\end{proof}

\subsection{Infinitely many indeterminates}\label{sec42}

Let $\kappa$ be a fixed infinite cardinal and $Y$ is a set of
cardinality $\kappa$. By Proposition~\ref{genlemma}, we have the
following result.

\begin{proposition}\label{generfi}
For each difference field $K$, every ideal of the ring $K\{Y\}$ is
generated by the set of cardinality $|\Sigma|\times |Y|$.
\end{proposition}

\begin{theorem}\label{theorefi}
Every difference field $K$ embeds into a $\kappa$-difference closed
field.
\end{theorem}
\begin{proof}

We should repeat the proof of Proposition~\ref{theorf} with some
modifications.

{\bf Construction}. The set of all (up to isomorphism) difference
integral domains generated over $K$ by the set of cardinality
$\kappa$  will be denoted by $\mathcal B$. We construct the same
partially ordered set $\mathbb S$ with elements
\[
\left(\mathop{\otimes}_{B\in\mathcal B'}\Qt(B),\mathfrak p_{\mathcal
B'}\right),
\]
where $\mathcal B'\subseteq \mathcal B$ is some subset in $\mathcal
B$ and  $\mathfrak p_{\mathcal B'}$ is a prime difference ideal in
$\otimes_{B\in \mathcal B'}\Qt(B)$. And we use the same order. By
Zorn's lemma, there exists a maximal element in $\mathbb S$. Let
\[
\left(\mathop{\otimes}_{B\in\overline{\mathcal B}}\Qt(B),\mathfrak
p_{\overline{\mathcal B}}\right)
\]
be a maximal element in this set. The residue field of the ideal
$\mathfrak p_{\overline{\mathcal B}}$ will be denoted by $K_1$. Now,
we will repeat this construction with the field $K_1$ instead of $K$
and will get a field $K_2$, then $K_3$ and so on. Using transfinite
induction, we define a field $K_\alpha$ for every ordinal $\alpha$.
Let $\gamma$ be the first cardinal greater than
$|\Sigma|\times\kappa$. The field $K_{\gamma}$ will be denoted by
$L$.

{\bf Difference closedness}. We should repeat the second part of the
proof of Theorem~\ref{theorf} using $\kappa$ instead of $\omega$ and
Proposition~\ref{generfi} instead of Proposition~\ref{genlemma}.
\end{proof}

\begin{remark}
It should be noted that, if the group $\Sigma$ is free, then we can
suppose that $L$ is algebraically closed in the previous theorem.
There are two ways of how to show this.

The first one, using the theorem, we obtain a field $L_0$ such that
$L_0$ is $\kappa$-closed. Let $L_0'$ be the algebraic closure of
$L_0$. Since $\Sigma$ is a free group, one can extend the action of
$\Sigma$ to the field $L_0'$. Then we repeat the procedure with
$L_0'$ instead of $K$ and obtain $L_1$ and its algebraic closure
$L_1'$ with some action of $\Sigma$. And, if $\gamma$ is the first
cardinal greater than $\kappa$, then $L_\gamma=L_\gamma'$ is the
desired field.

The second proof. We can repeat the proof of the previous theorem
with the following modifications. At each step, instead of the field
$K_\alpha$, we should choose its algebraic closure. Since $\Sigma$
is a free group, one can extend the action of $\Sigma$ to the
algebraic closure of $K_\alpha$.

The similar notice holds for difference closed fields, that is, if
the group $\Sigma$ is free, then every difference field embeds into
a difference closed field being algebraically closed. In
particulary, this is holds in the case of one automorphism, that is
$\Sigma=\mathbb Z$ (see~\cite{ChadzSur}).
\end{remark}

\section{Strong form of the Nullstellensatz for fields}\label{sec5}

\subsection{Finitely many indeterminates}\label{sec51}

We shall show that, for difference closed fields, there is a
bijection between the difference algebraic varieties and the perfect
difference ideals.

\begin{proposition}\label{prop_simple}
Let $K$ be a difference closed field and $A$ be a difference
finitely generated algebra over $K$ such that the zero ideal is the
only perfect ideal of $A$. Then $A$ coincides with $K$.
\end{proposition}
\begin{proof}

The algebra $A$ can be presented as follows
\[
A=K\{y_1,\ldots,y_n\}/\mathfrak p,
\]
where $\mathfrak p$ is a maximal perfect ideal. Since $K$ is
difference closed, the set $V(\mathfrak p)$ is not empty. Therefore,
there exists a point $x\in K^n$ such that $\mathfrak p\subseteq
I(x)$. But the ideal $I(x)$ is perfect. So, $\mathfrak p=I(x)$.
Consequently, $A$ coincides with $K$.
\end{proof}

\begin{theorem}\label{mainth}
Let $K$ be a difference closed field. Then, for every perfect
difference ideal  $\mathfrak t\subseteq K\{y_1,\ldots,y_n\}$, the
equality $\mathfrak t=I(V(\mathfrak t))$ holds.
\end{theorem}
\begin{proof}
The inclusion $\mathfrak t\subseteq I(V(\mathfrak t))$ holds by
definition. Let us show the other one. Consider the difference
algebra $A=K\{y_1,\ldots,y_n\}/\mathfrak t$. Let $x$ be an element
not belonging to the ideal $\mathfrak t$. Let $S$ be the
multiplicatively closed set generated by the elements $\sigma(x)$
for all $\sigma\in \Sigma$. Since $\mathfrak t$ is perfect, $S$ does
not meet the ideal $\mathfrak t$. Then the ring $S^{-1}A $ is a
nontrivial difference ring. By definition, $S^{-1}A=A\{1/t\}$ is a
difference finitely generated algebra over $K$. Let $\mathfrak n$ be
a maximal perfect ideal in the ring $S^{-1}A$. Then, from
Proposition~\ref{prop_simple}, if follows that $S^{-1}A/\mathfrak n$
coincides with $K$. Let $\mathfrak m$ be the contraction of
$\mathfrak n$ to $A$. Then $A/\mathfrak m$ coincides with $K$. Thus,
the images of the elements $y_1,\ldots,y_n$ define a point in $K^n$
such that all elements of $\mathfrak t$ vanish at this point but $x$
does not.
\end{proof}

The strong version of the difference Nullstellensatz has another
form.

\begin{theorem}\label{thm_closed_gen}
Let $K$ be a difference closed difference field. Then, for every
subsets $E,W\subseteq K\{y_1,\ldots,y_n\}$ such that $W$ is finite,
if there is a difference field $L$ containing $K$ and a point $x\in
L^n$ such that all elements of $E$ vanish at $x$ and $x$ is not a
zero for all elements of $W$, then there is an element $x'\in K^n$
with the same property.
\end{theorem}
\begin{proof}

Let $\mathfrak p$ be a prime difference ideal of
$K\{y_1,\ldots,y_n\}$ consisting of all polynomials vanishing at $x$
(so, $\mathfrak p$ contains $E$ and does not meet $W$) and
$A=K\{y_1,\ldots,y_n\}/\mathfrak p$. Let $S$ be the multiplicatively
closed set generated by the elements $\sigma(w)$ for all $\sigma\in
\Sigma$ and $w\in W$. Then $S$ is stable under the action of
$\Sigma$ and does not meet $\mathfrak p$. Thus, the ring $R=S^{-1}A$
is difference generated over $K$ by images of $y_i$ and elements
$1/w$, where $w\in W$. Therefore, $R$ is difference finitely
generated over $K$. Let $\mathfrak m$ be an arbitrary maximal
difference ideal of $R$, then $R/\mathfrak m$ coincides with $K$
because $K$ is difference closed. Thus, the images of $y_i$ in
$R/\mathfrak m$ give an element $s$ of $K^n$ such that every element
of $E$ vanishes at $s$ and $s$ is not a zero for all elements of
$W$.
\end{proof}

\subsection{Infinitely many indeterminates}\label{sec52}

In this case, the strong form of the Nullstellensatz also holds. Let
$\kappa$ be an infinite cardinal.

\begin{proposition}\label{prop_max}
Let $K$ be a $\kappa$-closed difference field and $A$ be a
difference $K$ algebra generated by the family $Y$ over $K$ such
that $|Y|\leqslant \kappa$ and the zero ideal of $A$ is the only
perfect difference ideal of $A$. Then $A$ coincides with $K$.
\end{proposition}
\begin{proof}

The algebra $A$ can be presented in the following form
\[
A=K\{Y\}/\mathfrak p,
\]
where $\mathfrak p$ is a maximal perfect ideal. Since $K$ is
$\kappa$-difference closed, $V(\mathfrak p)$ is not empty. Hence,
there exists a point $x\in K^Y$ such that $\mathfrak p\subseteq
I(x)$. But $I(x)$ is a prefect ideal, so, $\mathfrak p=I(x)$. And,
therefore, the ring $A$ coincides with $K$.
\end{proof}

\begin{theorem}\label{mainthi}
Let $K$ be a $\kappa$-closed difference field and let $Y$ be a set
such that $|Y|\leqslant\kappa$. Then, for every perfect difference
ideal $\mathfrak t\subseteq K\{Y\}$, the equality $\mathfrak
t=I(V(\mathfrak t))$ holds.
\end{theorem}
\begin{proof}

The inclusion $\mathfrak t\subseteq I(V(\mathfrak t))$ follows from
definition. Let $K\{Y\}/\mathfrak t$ be denoted by $A$ and $x$ is an
element not belonging to the ideal $\mathfrak t$. Let $S$ be the
multiplicatively closed set generated by the elements $\sigma(x)$
for all $\sigma\in \Sigma$. Then $S$ is stable under the action of
$\Sigma$ and does not meet the ideal $\mathfrak t$. The ring
$S^{-1}A$ is a nontrivial difference ring. By definition, we see
that $S^{-1}A=A\{1/t\}$. Consequently, $S^{-1}A$  is generated over
$K$ by the set $Y\cup\{1/t\}$ and we have
$|Y\cup\{1/t\}|\leqslant|Y|+1\leqslant\kappa+1=\kappa$. Let
$\mathfrak n$ be a maximal perfect ideal of the ring $S^{-1}A$.
Then, from Proposition~\ref{prop_max}, it follows that
$S^{-1}A/\mathfrak n$ coincides with $K$. Let $\mathfrak m$ be the
contraction of $\mathfrak n$ to $A$. Then $A/\mathfrak m$ also
coincides with $K$. The latter means that the images of elements of
$Y$ define a point in $K^Y$ such that all elements of $\mathfrak t$
vanish at this point but $t$ does not.
\end{proof}

As in the case of finitely many indeterminates, the strong version
of the difference Nullstellensatz has another form.

\begin{theorem}\label{thm_closed_gen_in}
Let $K$ be a $\kappa$-closed difference field and $Y$ is a set of
cardinality $\kappa$. Then, for every subsets $E,W\subseteq K\{Y\}$
such that $|W|\leqslant\kappa$, if there is a difference field $L$
containing $K$ and a point $x\in L^Y$ such that all elements of $E$
vanish at $x$ and $x$ is not a zero for all elements of $W$, then
there is an element $x'\in K^Y$ with the same property.
\end{theorem}
\begin{proof}

Let $\mathfrak p$ be a prime difference ideal of $K\{Y\}$ consisting
of all polynomials vanishing at $x$ (so, $\mathfrak p$ contains $E$
and does not meet $W$) and $A=K\{Y\}/\mathfrak p$. Let $S$ be the
multiplicatively closed set generated by the elements $\sigma(w)$
for all $\sigma\in \Sigma$ and $w\in W$. Then $S$ is stable under
the action of $\Sigma$ and does not meet $\mathfrak p$. Thus, the
ring $R=S^{-1}A$ is difference generated over $K$ by $Y\cup
\{1/w\mid w\in W\}$ and $|Y\cup W|\leqslant \kappa$. Let $\mathfrak
m$ be an arbitrary maximal difference ideal of $R$, then
$R/\mathfrak m$ coincides with $K$ because $K$ is $\kappa$-closed.
Thus, the image of $Y$ in $R/\mathfrak m$ gives an element $s$ of
$K^Y$ such that every element of $E$ vanishes on $s$ and $s$ is not
a zero for all elements of $W$.
\end{proof}

\section{The case of pseudofields}\label{sec6}

From this moment, we assume that the group $\Sigma$ is infinite. The
case of finite group is studied in~\cite{Tr5}.

\subsection{Finitely many indeterminates}\label{sec61}

We will prove some auxiliary statements before proving the main
result.

\begin{proposition}\label{specst}
Let $\gamma_\Sigma$ be the cardinal coinciding with the number of
all maximal filters on $\Sigma$. Then, for every pseudofield $L$, we
have $|\Spec L|\leqslant \gamma_\Sigma$.
\end{proposition}
\begin{proof}
Let $\mathfrak m$ be a prime ideal of $L$, then the corresponding
quotient field will be denoted by $K$ and $\pi\colon L\to K$ will be
the quotient homomorphism. Then, from Proposition~\ref{taylor}, it
follows that there exists a difference homomorphism $\Phi\colon L\to
\Fun K$ such that $\Phi(a)(e)=\pi(a)$, where $e$ is the identity of
$\Sigma$. Since $L$ is a pseudofield, $\Phi$ is injective. Moreover,
since $L$ is absolutely flat, the map $\Spec \Fun K\to \Spec L$ is
surjective~\cite[Chapter~3, Exercise~29 and Exercise~30]{AM}.
Consequently, $|\Spec L|\leqslant|\Spec \Fun K|$. But all prime
ideals of $\Fun K$ correspond to all maximal filters on $\Sigma$.
So, we have the desired result.
\end{proof}

\begin{proposition}\label{simp}
Every simple difference ring embeds into a pseudofield.
\end{proposition}
\begin{proof}

Let $R$ be a simple difference ring and $\mathfrak m$ be its maximal
ideal. Then the residue field of this ideal will be denoted by $K$.
Let $\pi\colon R\to K$ be the quotient homomorphism. Then, from
Proposition~\ref{taylor}, it follows that there exists a difference
homomorphism $\Phi\colon R\to \Fun K$ such that $\Phi(a)(e)=\pi(a)$,
where $e$ is the identity of $\Sigma$. The ring $\Fun K$ is an
absolutely flat difference ring and let $\mathfrak n$ be a maximal
difference ideal in $\Fun K$. Then $L=\Fun K/\mathfrak n$ is simple
and absolutely flat~\cite[Section~2, Exercis~28]{AM}, thus, a
pseudofield. The composition $R\to\Fun K\to L$ gives us the desired
embedding.
\end{proof}

\begin{proposition}\label{tens}
Let $A$ be a pseudofield, and $B$ and $C$ are nonzero difference
algebras over $A$. Then the difference algebra $B\otimes_{A}C$  is
nonzero.
\end{proposition}
\begin{proof}
Let $\mathfrak q$ be a prime ideal of $B$, then its contraction to
$A$ will be denoted by $\mathfrak p$. Since $A$ is a simple
difference ring, $A$ is a subring of $C$. Again, since $A$ is
absolutely flat, the contraction map $\Spec C \to \Spec A$ is
surjective~\cite[Chapter~3, Exercise~29 and Exercise~30]{AM}).
Therefore, there exists a prime ideal $\mathfrak q'$ in $C$
contracting to $\mathfrak p$.  We will prove that the ring
\[
R=(B\mathop{\otimes}_A C)_{\mathfrak p}= B_\mathfrak
p\mathop{\otimes}_{A_\mathfrak p} C_{\mathfrak p}
\]
is not zero. For that we should show that $B_{\mathfrak p}$ and
$C_\mathfrak p$ are not zero.

The prime ideals of $B_\mathfrak p$ corresponds to the preimage of
$\mathfrak p$ for the natural map $\Spec B\to \Spec A$. By
definition, the fibre is not empty. Therefore, $\Spec B_\mathfrak p$
is not empty and, thus, $B_\mathfrak p$ is not zero. In the same
manner, we prove that $C_\mathfrak p$ is not empty. Since
$A_{\mathfrak p}$ is a field and algebras $B_{\mathfrak p}$ and
$C_\mathfrak p$ are not zero, the ring $R$ is not the zero ring.
\end{proof}

We will say that a partially ordered set $S$ is $\kappa$-Noetherian,
where $\kappa$ is an infinite cardinal, if every strictly ascending
chain of elements of $S$ is of cardinality less than or equal to
$\kappa$. Precisely, for every well-ordered set $\mathcal A$ of
cardinality strictly greater than $\kappa$, every monotone function
$\varphi\colon A\to S$ is stationary.

\begin{proposition}\label{poset}
Let $\{S_\alpha\}_{\alpha\in\mathcal A}$ be a family of partially
ordered sets and, for every $\alpha$, the set $S_\alpha$ is
$\kappa$-Noetherian. Then the partially ordered set
$\prod_{\alpha\in \mathcal A}S_\alpha$ is $|\mathcal
A|\times\kappa$-Noetherian.
\end{proposition}
\begin{proof}

Let $\{x_\gamma\}$  be a strictly ascending chain of
$S=\prod_{\alpha\in \mathcal A}S_\alpha$ and $\pi_\alpha\colon S\to
S_\alpha$ be the projections to the corresponding factors. We
produce the following family of sets
\[
T_\alpha=\{\,s\in S_\alpha\mid\exists x_\gamma\colon
s=\pi_\alpha(x_\gamma)\,\}.
\]
Since $S_\alpha$ is $\kappa$-Noetherian, the cardinality of
$T_\alpha$ is less than or equal to $\kappa$. We set
$X=\sqcup_{\alpha}T_\alpha$. Then, for every element  $x_\gamma$, we
define the subset $X_\gamma$ in $X$ by the following rule
\[
X_\gamma=\bigsqcup_\alpha\{\,s\in T_\alpha\mid
s\leqslant\pi_\alpha(x_\gamma) \,\}.
\]
We see that the family $\{X_\gamma\}$ is a strictly ascending chain
in the set of all subsets of $X$. So, the cardinality of this chain
is not greater than the cardinality of $X$, but $|X|=|\mathcal
A|\times\kappa$.
\end{proof}

\begin{proposition}\label{noethprop}
Let $\gamma_\Sigma$ be the cardinal coinciding with the number of
all maximal filters on $\Sigma$. Then, for every pseudofield $K$,
the set $\PRad K\{y_1,\ldots,y_n\}$ is
$|\Sigma|\times\gamma_\Sigma$-Noetherian.
\end{proposition}
\begin{proof}

We will prove a more general statement. Namely, we will show that
\[
\Rad K\{y_1,\ldots,y_n\}
\]
is $|\Sigma|\times\gamma_\Sigma$-Noetherian. Let $\mathfrak
m_\alpha$ be all prime ideals of $K$, the corresponding residue
fields will be denoted by $F_\alpha$. Consider the rings
\[
R_\alpha=R_{\mathfrak m_\alpha}=F_\alpha[\ldots,\sigma y_i,\ldots].
\]
Then, for every radical ideal $\mathfrak t$, we define the family of
ideals $\{\,(\mathfrak t)_{\mathfrak m_\alpha}\,\}$. This family is
an element of the partially ordered set
\[
\prod_{\alpha}\Rad(R_\alpha).
\]

The map $\mathfrak t \mapsto \{\,(\mathfrak t)_{\mathfrak
m_\alpha}\,\}$ preserves the order, i.e., from the inclusion
$\mathfrak t\subseteq\mathfrak t'$, it follows that, for every
$\alpha$, we have $(\mathfrak t)_{\mathfrak m_\alpha}\subseteq
(\mathfrak t')_{\mathfrak m_\alpha}$. Note that different radical
ideals give different families. Indeed, let $\mathfrak
t\neq\mathfrak t'$. Then there exists an element $x\in\mathfrak
t\setminus\mathfrak t'$ (or $x\in\mathfrak t'\setminus\mathfrak t$).
Consequently, there is a prime ideal $\mathfrak p$ containing
$\mathfrak t'$ and not containing $x$. Then $(\mathfrak
t')_{\mathfrak p}\neq (1)$ but $(\mathfrak t)_{\mathfrak p}=(1)$.

Hence, the partially ordered set $\Rad K\{y_1,\ldots,y_n\}$ can be
embedded into the partially ordered set
$\prod_{\alpha}\Rad(R_\alpha)$. The set $\Id R_{\alpha}$ is
$\Sigma$-Noetherian by Proposition~\ref{genlemma}. Therefore, its
subset $\Rad R_{\alpha}$ is also $\Sigma$-Noetherian. Then the
desired result follows from Proposition~\ref{specst} and
Proposition~\ref{poset}.
\end{proof}

\begin{corollary}\label{generatp}
Let $\gamma_\Sigma$ be the cardinal coinciding with the number of
all maximal filters on $\Sigma$. Then, for every pseudofield $K$,
every radical difference ideal of the ring $K\{y_1,\ldots,y_n\}$ can
be presented as $\mathfrak r([E])$, where $|E|\leqslant
|\Sigma|\times\gamma_\Sigma$.
\end{corollary}

\begin{theorem}\label{theorps}
Every pseudofield $K$ embeds into a difference closed pseudofield.
\end{theorem}
\begin{proof}

{\bf Construction}. Consider the family $\{\,B_\alpha\,\}$ of all
(up to isomorphism) simple difference finitely generated algebras
over $K$. The ring $\otimes_{\alpha}B_\alpha$ will be denoted by $R$
(the tensor products are taken over $K$). From
Proposition~\ref{tens}, it follows that $R$ is a nonzero difference
ring. Let $\mathfrak m$ be a maximal difference ideal in $R$. Then,
by Proposition~\ref{simp}, it follows that the ring $R/\mathfrak m$
can be embedded into a pseudofield $K_1$. Using transfinite
induction, we obtain pseudofields $K_\alpha$ for every ordinal
$\alpha$. Let $\gamma$ be the first cardinal greater than
$|\Sigma|\times \gamma_\Sigma$. Let us show that $L=K_{\gamma}$ is
difference closed.

{\bf Difference closedness}. Let $\mathfrak b$ be a difference ideal
in $L\{y_1,\ldots,y_n\}$. Then it is contained in a maximal
difference ideal $\mathfrak t$. The ring
$L\{y_1,\ldots,y_n\}/\mathfrak t$ will be denoted by $B$.

From corollary~\ref{generatp}, it follows that the ideal $\mathfrak
t$ can be presented as  $\mathfrak r([E])$, where
$|E|\leqslant|\Sigma|\times\gamma_\Sigma$. The ideal $[E]$ will be
denoted by $\mathfrak a$ and the ring
\[
L\{y_1,\ldots,y_n\}/\mathfrak a
\]
by $B'$. The nilradical of  $B'$ is the unique pseudoprime ideal of
$B'$. Since cardinality of $E$ is less than $\gamma$, there exists
an intermediate pseudofield $K_\alpha$ containing all coefficients
of all elements of $E$.  Then
\[
B'=L\mathop{\otimes}_{K_\alpha}K_\alpha \{y_1,\ldots,y_n\}/[E].
\]
But there is a common zero for $E$ in  $K_{\alpha+1}\subseteq L$.
Therefore there exists a difference homomorphism
\[
K_\gamma \{y_1,\ldots,y_n\}/[E]\to L.
\]
Consequently, a difference homomorphism of $B'$ into $L$. Since $L$
is pseudofield, the kernel of the latter homomorphism is pseudoprime
and, thus, the homomorphism factors through $B$. Therefore, we have
a difference homomorphism $B\to L$. The images of elements
$y_1,\ldots,y_n$ define a common zero for $\mathfrak b$ in $L^n$.
\end{proof}

\subsection{Infinitely many indeterminates}\label{sec62}

In this section, we suppose that $\kappa$ is an infinite cardinal
and $Y$ is a set of cardinality $\kappa$.

\begin{proposition}\label{prop_noeth}
Let $\gamma_\Sigma$ be the cardinal coinciding with the number of
all maximal filters on $\Sigma$. Then, for every pseudofield $K$,
the set $\PRad K\{Y\}$ is $|\Sigma\times
Y|\times\gamma_\Sigma$-Noetherian.
\end{proposition}
\begin{proof}

We will show that the set $\Rad K\{Y\}$ is $|\Sigma\times
Y|\times\gamma_\Sigma$-Noetherian. Let $\mathfrak m_\alpha$ be all
the prime ideals of $K$. The corresponding residue fields will be
denoted by $F_\alpha$. Let the ring
\[
R_{\mathfrak m_\alpha}=F_\alpha[\ldots,\sigma y,\ldots].
\]
be denoted by $R_\alpha$. Then every radical ideal $\mathfrak t$
gives the family of ideals $\{\,(\mathfrak t)_{\mathfrak
m_\alpha}\,\}$. This family is an element of the partially ordered
set
\[
\prod_{\alpha}\Rad(R_\alpha).
\]

In the proof of Proposition~\ref{noethprop}, we obtained that the
map $\mathfrak t\mapsto \{\,(\mathfrak t)_{\mathfrak m_\alpha}\,\}$
is injective and preserves the inclusions. Therefore, the partially
ordered set $\Rad K\{Y\}$ embeds into the set
$\prod_{\alpha}\Rad(R_\alpha)$. The set $\Id R_{\alpha}$ is
$|\Sigma\times Y|$-Noetherian. Consequently, its subset  $\Rad
R_{\alpha}$ is also $|\Sigma\times Y|$-Noetherian. Then the desired
result follows from Propositions~\ref{specst} and~\ref{poset}.
\end{proof}

\begin{corollary}\label{corollkappa}
Let $\gamma_\Sigma$ be the cardinal coinciding with the number of
all maximal filters on $\Sigma$. Then, for every pseudofield $K$,
every radical difference ideal of $K\{Y\}$ can be presented as
$\mathfrak r([E])$, where $|E|\leqslant |\Sigma\times
Y|\times\gamma_\Sigma$.
\end{corollary}

\begin{theorem}\label{theorii}

Every pseudofield $K$ embeds into a $\kappa$-closed pseudofield.
\end{theorem}
\begin{proof}

{\bf Construction}. Consider the family $\{\,B_\alpha\,\}$
consisting of all (up to isomorphism) simple difference $K$-algebras
generated by a set of cardinality $\kappa$. As in the proof of
Proposition~\ref{theorps}, we take a maximal difference ideal
$\mathfrak m$ in $\otimes_{\alpha}B_\alpha$. Then the ring
$(\otimes_{\alpha}B_\alpha)/\mathfrak m$ embeds into a pseudofield
$K_1$. Repeating the process, we obtain a pseudofield $K_\alpha$ for
all ordinals $\alpha$. We will show that $K_\gamma$ is
$\kappa$-closed if $\gamma$ is the first cardinal greater than
$|\Sigma|\times \gamma_\Sigma$.

{\bf Difference closedness}. As in the proof of
Theorem~\ref{theorps}, we consider a maximal difference ideal
$\mathfrak t$ in $L\{Y\}$ and $B=L\{Y\}/\mathfrak t$.

Using corollary~\ref{corollkappa} instead of~\ref{generatp}, we have
$\mathfrak t = \mathfrak r[E]$, where
$|E|\leqslant|\Sigma|\times\gamma_\Sigma\times \kappa$. In the same
manner, we get
\[
B=L\{Y\}/\mathfrak a=L\mathop{\otimes}_{K_\alpha}K_\alpha \{Y\}/[E].
\]
Since there is a common zero for $E$ in  $K_{\alpha+1}\subseteq L$,
we have $B=L$. Whence, the images of elements of $Y$ define a common
zero for $\mathfrak b$ in $L^Y$.
\end{proof}

\bibliographystyle{plain}
\bibliography{full_bib}

\begin{thebibliography}{10}

\bibitem{AmMasPV}
K.~Amano and A.~Masuoka.
\newblock {P}icard-{V}essiot extensions of artinian simple module algebras.
\newblock {\em J. {A}lgebra}, 285:743--767, 2005.

\bibitem{AOT}
B.~Antieau, A.~Ovchinnikov, and D.~Trushin.
\newblock Galois theory of difference equations with periodic parameters.
\newblock Submitted for publication, \url{http://arxiv.org/abs/1009.1159},
  2010.

\bibitem{AM}
M.~F. Atiyah and I.~G. Macdonald.
\newblock {\em Introduction to commutative algebra}.
\newblock {A}ddison-{W}esley, 1969.

\bibitem{ChadzSur}
Z.~Chatzidakis.
\newblock A survey on the model theory of difference fields.
\newblock {\em Model theory, algebra, and geometry}, 39:65--96, 2000.

\bibitem{ChatHarSinDGG}
Z.~Chatzidakis, C.~Hardouin, and M.~Singer.
\newblock On the definitions of difference {G}alois groups.
\newblock In {\em Model Theory with Applications to Algebra and Analysis.
  Volume 1}, volume 349 of {\em London {M}athematical {S}ociety {L}ecture
  {N}ote {S}eries}, pages 73--110. Cambridge {U}niversity {P}ress, 2008.

\bibitem{ChatzHrushMDF}
Z.~Chatzidakis and E.~Hrushovski.
\newblock Model theory of difference fields.
\newblock {\em {T}rans. {A}mer. {M}ath. {S}oc.}, 351:2997--3071, 1999.

\bibitem{CohnDA}
R.~M. Cohn.
\newblock {\em Difference algebra}.
\newblock Number~17. Interscience {P}ublishers {J}ohn {W}iley \& {S}ons, New
  {Y}ork-{L}ondon-{S}ydeny, 1965.

\bibitem{HarSinLDiff}
C.~Hardouin and M.~F. Singer.
\newblock Differential {G}alois theory of linear difference equations.
\newblock {\em Mathematische {A}nnalen}, 342(2):333--377, 2008.

\bibitem{HrushvN}
E.~Hrushovski and F.~Point.
\newblock On von {N}eumann regular rings with an automorphism.
\newblock {\em Journal of Algebra}, 315(1):76--120, 2007.

\bibitem{LevinDA}
A.~Levin.
\newblock {\em Difference algebra}, volume~8 of {\em Algebra and
  {A}pplications}.
\newblock Springer, New {Y}ork, 2008.

\bibitem{MorikDG1}
S.~Morikawa.
\newblock On a general difference {G}alois theory {I}.
\newblock {\em Annales de {I}'{I}nstitut {F}ourier}, 59(7):2709--2732, 2009.

\bibitem{UmMorDG2}
S.~Morikawa and H.~Umemura.
\newblock On a general difference {G}alois theory {II}.
\newblock {\em Annales de {I}'{I}nstitut {F}ourier}, 59(7):2733--2771, 2009.

\bibitem{BPoizotCMT}
B.~Poizat.
\newblock {\em A Course in Model Theory: An Introduction to Contemporary
  Mathematical Logic}.
\newblock {S}pringer-{V}erlag, New York, 2000.

\bibitem{TakeuHopPV}
M.~Takeuchi.
\newblock A {H}opf algebraic approach to the {P}icard-{V}essiot theory.
\newblock {\em J. {A}lgebra}, 122:481--509, 1989.

\bibitem{Tr5}
D.~Trushin.
\newblock Difference {N}ullstellensatz in the case of finite group.
\newblock Submitted for publication, \url{http://arxiv.org/abs/0908.3863},
  2010.

\bibitem{vdPSinDcE}
M.~van~der Put and M.~F. Singer.
\newblock {\em Galois theory of difference equations}, volume 1666 of {\em
  Lecture Notes in Mathematics}.
\newblock {Springer-{V}erlag}, Berlin, 1997.

\bibitem{vdPSinLDcE}
M.~van~der Put and M.~F. Singer.
\newblock {\em Galois theory of linear differential equations}, volume 328 of
  {\em Grundlehren der Mathematischen Wissenschaften {[Fundamental} Principles
  of Mathematical Sciences]}.
\newblock {Springer-Verlag}, Berlin, 2003.

\bibitem{WibmerDcChev}
M.~Wibmer.
\newblock A {C}hevalley theorem for difference equations.
\newblock \url{http://arxiv.org/abs/1010.5066}, 2010.

\bibitem{WibmerThesis}
M.~Wibmer.
\newblock {\em Geometric Difference {G}alois Theory}.
\newblock PhD thesis, Heidelberg, 2010.

\end{thebibliography}

\end{document}